\documentclass[11pt,letter]{amsart}
\usepackage{amsfonts,amssymb,amsthm,amsmath,amsxtra,amscd,verbatim,eucal}
\usepackage[all]{xy}
\usepackage[dvips]{graphics}

\theoremstyle{plain}
\newtheorem{thm}{Theorem}[section]
\newtheorem{lem}[thm]{Lemma}
\newtheorem{prop}[thm]{Proposition}
\newtheorem{cor}[thm]{Corollary}

\theoremstyle{definition}
\newtheorem{defi}[thm]{Definition}

\theoremstyle{remark}
\newtheorem{eg}[thm]{Example}

\newtheorem{rmk}[thm]{Remark}

\def\Z{{\mathbb Z}}

\def\C{{\mathbb C}}

\def\R{{\mathbb R}}
\def\Q{{\mathbb Q}}

\def\K{\mathcal{K}}
\def\O{\mathcal{O}}
\def\LL{\mathcal{L}}

\def\W{\mathcal{W}}
\def\J{\mathcal{J}}
\def\I{\mathcal{I}}

\def\a{\mathfrak{a}}

\def\d{\delta}
\def\f{\phi}
\def\ff{\psi}

\def\ep{\epsilon}

\def\l{\lambda}
\def\n{\nu}
\def\m{\mu}

\def\D{\Delta}

\def\o{\circ}

\def\rat{\dashrightarrow}

\def\lrd{\lfloor}
\def\rrd{\rfloor}
\def\lru{\lceil}
\def\rru{\rceil}

\def\.{\cdot}
\def\({\Big{(}}
\def\){\Big{)}}
\def\^{\widehat}
\def\~{\widetilde}

\renewcommand{\and}{ \quad \text{and} \quad }
\renewcommand{\for}{ \quad \text{for} \ \, }
\newcommand{\fall}{ \quad \text{for all} \ \, }

\DeclareMathOperator{\val} {val}

\DeclareMathOperator{\adj} {adj}
\DeclareMathOperator{\Ex} {Ex}

\DeclareMathOperator{\ord} {ord}

\DeclareMathOperator{\Div} {Div}
\DeclareMathOperator{\divisor} {div}

\DeclareMathOperator{\Supp} {Supp}

\DeclareMathOperator{\lc} {lc}

\title{Singularities on normal varieties}

\author{Tommaso de Fernex}
\address{Department of Mathematics, University of Utah, 155 South 1400 East,
Salt Lake City, UT 48112-0090, USA}
\email{defernex@math.utah.edu}

\author{Christopher D. Hacon}
\address{Department of Mathematics, University of Utah, 155 South 1400 East,
Salt Lake City, UT 48112-0090, USA}
\email{hacon@math.utah.edu}

\thanks{The first author was partially
supported by NSF research grant no: 0548325.
The second author was partially supported by NSF research grant no: 0456363
and an AMS Centennial fellowship.}


\begin{document}

\begin{abstract}
In this paper we generalize the definitions of singularities of pairs
and multiplier ideal sheaves to pairs on arbitrary normal varieties,
without any assumption on the variety being $\Q$-Gorenstein or the
pair being log $\Q$-Gorenstein.
The main features of the theory extend to this setting in a natural way.
\end{abstract}

\maketitle

\section{Introduction}

The theory of
singularities of pairs and multiplier ideal sheaves has become a core part
of the study of higher dimensional algebraic varieties
(e.g., see \cite{Kol,KM,Laz,EM2} for an overview
of the theory and various applications). In fact,
pairs naturally arise in a geometrically meaningful way in a variety
of instances: as boundaries of open varieties,
markings on varieties in moduli problems, discriminants
and orbifold structures of morphisms, base schemes of rational maps,
and inductive tools in higher dimensional geometry.

The main purpose of this paper is to investigate possible extensions of
the theory to settings which are more general than the ones in which
it has been introduced and studied.
Our priority, naturally, is to perform this generalization in such a way
that the essential features are preserved.

Given a $\Q$-Gorenstein variety $X$, several invariants
have been defined via resolution of singularities.
A key ingredient in their definition is the
relative canonical divisor of a resolution $f\colon Y \to X$,
that is, the exceptional $\Q$-divisor $K_{Y/X} := K_Y - f^*K_X$
(here we fix $K_Y$ so that $f_*K_Y = K_X$).
The difficulty in extending the definitions of such invariants
to arbitrary normal varieties arises as soon as $K_X$ is not $\Q$-Cartier,
as it is unclear in this case what should be its pullback.
One way around the problem is to perturb $K_X$
by adding a {\it boundary}, that is,
and effective $\Q$-divisor $\D$ such that $K_X + \D$ is $\Q$-Cartier.
This also gives rise to a pair $(X,\D)$, but the boundary
itself may have no particular geometric meaning, and
it is not clear a priori that there exists a natural choice for $\D$.

Our approach to the problem is different and more direct.
We introduce a notion of pullback of (Weil) $\Q$-divisors
which agrees with the usual one for $\Q$-Cartier $\Q$-divisors.
In this way we are able to define {\it relative canonical divisors}
$K_{Y/X} := K_Y + f^*(-K_X)$ and $K_{Y/X}^- := K_Y - f^*K_X$
for any proper birational morphism $f \colon Y \to X$ of normal varieties.
These are exceptional $\R$-divisors that coincide when $K_X$ is $\Q$-Cartier,
but may be different otherwise.
We also define a suitable approximation of $K_{Y/X}^-$
via $\Q$-divisors $K_{m,Y/X}$ (for $m \ge 1$), that we call
{\it limiting relative canonical divisors}.
Using these notions, we generalize the definitions of multiplier
ideals and singularities of pairs to pairs of the form
$(X,Z)$, where $X$ is an arbitrary normal variety and $Z = \sum b_k\.Z_k$ is
an effective formal linear combination of proper closed subschemes of $X$.

The {\it multiplier ideal sheaf} $\J(X,Z)$ of $(X,Z)$ is defined,
in our generality,
as the unique maximal element in the collection of ideal sheaves
$$
\big\{ (f_m)_*\O_{Y_m}(\lru K_{m,Y_m/X} - f_m^{-1}(Z)\rru) \big\}_{m \ge 1},
$$
where for every $m$ the morphism $f_m \colon Y_m \to X$ is a
`high enough' log resolution of $(X,Z)$ depending on $m$.
The core result of the paper is that $\J(X,Z)$ can be
realized as the multiplier ideal sheaf of a suitable log $\Q$-Gorenstein pair.

\begin{thm}
For any pair $(X,Z)$ as above, there is a boundary $\D$ on $X$ such that
$$
\J(X,Z) = \J((X,\D);Z).
$$
\end{thm}

In particular, we deduce the surprising fact that the set of ideal sheaves
$$
\{\J((X,\D);Z) \mid \text{$\D$ is a boundary on $X$} \}
$$
has a unique maximal element, namely $\J(X,Z)$. A posteriori, one can take this
maximal element as the definition of $\J(X,Z)$.
Using this result, all the main properties related to multiplier ideals,
such as vanishing theorems, connectedness properties,
and basic inversion of adjunction statements, extend immediately to
the general setting.

In order to generalize the notions of
{\it log terminal} and {\it log canonical singularities},
we impose log discrepancy conditions with respect to the
limiting relative canonical divisors $K_{m,Y/X}$.
In a similar vein, we have the following result.

\begin{thm}
A pair $(X,Z)$ is log terminal (resp., log canonical)
if and only if there is a boundary $\D$ on $X$ such that
$((X,\D);Z)$ is log terminal (resp., log canonical).
\end{thm}

We immediately deduce, for instance, that as
in the $\Q$-Gorenstein case a normal variety with
log terminal singularities (resp., with Cohen--Macaulay log
canonical singularities) has rational singularities
(resp., Du Bois singularities).
Kawamata's subadjunction theorem is also generalized to our context.
In fact, we observe that minimal log canonical centers
(which in general are not known to be $\Q$-Gorenstein)
are log terminal; this provides in particular a natural setting
for the theory developed in this paper.
Finally, we check that in dimension two our notions of log terminal
and log canonical singularities agree with those of numerically log terminal
and numerically log canonical singularities,
which in particular implies that they are always $\Q$-Gorenstein.

By contrast, our definition of {\it terminal} and {\it canonical singularities} uses
log discrepancy conditions with respect to the
relative canonical divisor $K_{Y/X}$.
When $Z$ is a $\Q$-Cartier $\Q$-divisor, we extend to this setting the
following characterization of canonical singularities.

\begin{prop}
If $Z$ is a $\Q$-Cartier $\Q$-divisor, then $(X,Z)$ is canonical if and only
if for any sufficiently divisible $m \ge 1$ and
for every sufficiently high log resolution $f \colon Y \to X$
there is an inclusion $\O_X(m(K_X + Z))\.\O_Y \subseteq \O_Y(m(K_Y + Z_Y))$
as sub-$\O_X$-modules of the constant sheaf of rational functions,
where $Z_Y$ is the proper transform of $Z$.
\end{prop}

Using this property, the main features of canonical singularities, such as
the deformation invariance properties of
plurigenera (for singular varieties of general type),
of canonical singularities and of numerical Kodaira dimension
easily extend to the more general setting.

We expect that the larger freedom in defining these notions of singularities
should have interesting applications. In the (log) $\Q$-Gorenstein setting
many applications rely on multiplier ideals and
their vanishing theorems, and it is encouraging that these powerful
methods extend to our setting.

The original motivation of this research comes from
a question posed by Valery Alexeev during the AIM Workshop \cite{AIM06},
which asks whether it is possible to
generalize the definitions of singularities of pairs in a wider context than
the usual one.
The question itself was motivated by an example, due to Paul Hacking, of
a flat family of pairs $(S_t,D_t)$, where $S_t$ is a smooth surface
and $D_t$ is an effective divisor, that specializes to
a pair $(S_0,D_0)$, where $S_0$ is a singular surface and
the ideal sheaf of $D_0$
acquires an embedded prime at the singularity of $S_0$.

The example brings to light an important issue: namely
that often, in the literature, pairs $(X,Z)$ have
been intended in a combined way, both geometrically (as in one of the
situations previously described) and as a correction to the possible failure
of $K_X$ being $\Q$-Cartier---by incorporating
a boundary $\D$ into $Z$, so to speak.
We insist in this paper to keep the two things separated.

The question of defining multiplier ideals in the generality
treated in this paper arises naturally also in connection with
the {\it generalized test ideal} introduced by Hara and Yoshida
\cite{HY} (see also \cite{HT})
using the Frobenius action in positive characteristics, as the latter
can be defined without any (log) $\Q$-Gorenstein assumption.
In the (log) $\Q$-Gorenstein setting,
multiplier ideals reduce, for sufficiently large characteristics,
to the corresponding generalized test ideals
(see \cite{Smi00,Har01,HY,Tak04}).
It follows by independent results of Hara and Blickle (see \cite{Bli})
that, in the toric setting,
the same happens without any (log) $\Q$-Gorenstein assumption
for the multiplier ideals defined in this paper. It would be interesting to
see if this property holds in general; this question was raised by Hara.

In the first two sections of the paper we work over an arbitrary field;
starting from Section~\ref{sect:multiplier-ideals}
we will restrict the setting to varieties over an
algebraically closed field of characteristic zero.
A {\it divisor} on a normal variety $X$ will be understood to be a Weil divisor,
unless otherwise specified.

\subsection{Acknowledgments}
The authors would like to thank Lawrence Ein
and Karen Smith for useful conversations,
and Manuel Blickle, Nobuo Hara, Karl Schwede and Shunsuke Takagi
for useful comments.
The authors are very grateful to the referee for
many valuable suggestions, comments and corrections.

\section{Valuations of $\Q$-divisors}

Let $X$ be a normal variety.
A {\it divisorial valuation} $v$ on $X$ is a discrete valuation
of the function field of $X$ of the form $v = q\val_F$
where $q \in \Z_+$ and $F$ is a prime divisor {\it over} $X$,
that is, on a normal variety $X'$ with a given birational morphism
$\m \colon X' \to X$.

Throughout this section, we fix a divisorial
valuation $v$ of $X$.
If $D$ is a Cartier divisor on $X$, then the
valuation $v(D)$ of $D$ is given by
$q$ times the coefficient of $F$ in the divisor $\m^*D$.
The valuation $v(Z)$ of a proper closed
subscheme $Z \subset X$ is given by
$$
v(Z) = v(\I_Z) :=
\min \{ v(\f) \mid \f \in \I_Z(U), \; U \cap c_X(v) \ne \emptyset \},
$$
where $\I_Z \subseteq \O_X$ is the ideal sheaf of $Z$.
This definition extends to formal $\R$-linear combinations
$\sum a_k\.Z_k$ of proper closed subschemes $Z_k \subset X$
by setting $v(\sum a_k\.Z_k) := \sum a_k\.v(Z_k)$.

More generally, let $\I \subset \K$ be a
finitely generated sub-$\O_X$-module of the
constant sheaf of rational functions $\K= \K_X$ on $X$.
For short, we will refer to $\I$ as a {\it (coherent) fractional
ideal sheaf} on $X$.

\begin{defi}
The {\it valuation} $v(\I)$
of a non-zero fractional ideal sheaf $\I \subset \K$ along $v$ is given by
$$
v(\I) :=
\min \{ v(\f) \mid \f \in \I(U), \; U \cap c_X(v) \ne \emptyset \}.
$$
The {\it valuation} $v(I)$
of a formal linear combination $I = \sum a_k \.\I_k$ of
fractional ideal sheaves $\I_k \subset \K$ along $v$ is defined by
$v(I) := \sum a_k \. v(\I_k)$.
\end{defi}

If $\I$ and $\I'$ are fractional ideal sheaves on $X$
with $\I \subseteq \I'$, then $v(\I ) \ge v(\I' )$.
In the case of ideal sheaves, this definition
coincides with the one previously given, and
if $D$ is a Cartier divisor, then $v(D) = v(\O_X(-D))$.

Consider now an arbitrary divisor $D$ on $X$.

\begin{defi}
The {\it $\natural$-valuation} (or {\it natural valuation})
along $v$ of a divisor $D$ on $X$ is
$$
v^\natural(D) := v(\O_X(-D)).
$$
\end{defi}

Clearly, we have $v^\natural(D) = v(D)$ for any Cartier divisor $D$.
Note also that, as $\O_X(D)\.\O_X(-D) \subseteq \O_X$,
we have that $v^\natural(-D) + v^\natural(D) \ge 0$.

In general the $\natural$-valuation of divisors is not
linear with respect to the group structure of $\Div(X)$,
as the next example shows.

\begin{eg}\label{eg:non-additivity:natural-val}
Let $X = \{ xy = z^2 \} \subset \C^3$, and let $v = \val_E$, where $E$
is the exceptional divisor of the blow up of $X$ at the origin. Then, for
any two lines $L,M \subset X$ (passing through the origin), we have
$v^\natural(L) = v^\natural(M) = v^\natural(L + M) = 1$,
and thus $v^\natural(L+M) \ne v^\natural(L) + v^\natural(M)$.
In particular, $v^\natural(2L) = v^\natural(L)$.
Note also that $v^\natural(-L) = 0$.
\end{eg}

\begin{lem}\label{lem:natural-val:additivity:special-case}
Let $C$ be a Cartier divisor on $X$. Then
$v^\natural(C+ D) = v(C) + v^\natural(D)$
for every divisor $D$ on $X$.
\end{lem}

\begin{proof}
Since $\O_X(-C)$ is locally generated by one rational function,
one can check that $\O_X(-C-D)=\O_X(-C)\.\O_X(-D)$,
and the assertion follows.
\end{proof}

\begin{defi}
To any non-trivial fractional ideal sheaf $\I$ on $X$,
we associate the divisor
$$
\divisor(\I) := \sum_{E\subset X} \val_E(\I)\.E,
$$
where the sum is taken over all prime divisors $E$ on $X$.
Equivalently, $\divisor(\I)$ is the divisor on $X$ for which
$\O_X(-\divisor(\I)) = \I^{\vee\vee}$.
In particular, $\divisor(\O_X(-D)) = D$ for any divisor $D$.
We call $\divisor(\I)$ the {\it divisorial part} of $\I$.
\end{defi}

Consider now a birational morphism $f\colon Y \to X$ from
a normal variety $Y$.

\begin{defi}
For any divisor $D$ on $X$, the {\it $\natural$-pullback}
(or {\it natural pullback}) of $D$ to $Y$ is given by
$$
f^\natural D := \divisor(\O_X(-D)\.\O_Y).
$$
In other words, $f^\natural D = \sum \val_E^\natural(D)\.E$,
where the sum is taken over all prime divisors $E$ on $Y$.
In particular, $\O_Y(-f^\natural D) = (\O_X(-D)\.\O_Y)^{\vee\vee}$.

\end{defi}

\begin{lem}\label{prop:composition-natural-pullback}
Let $f\colon Y \to X$ and $g\colon V \to Y$
be two birational morphisms of normal varieties.
Then, for every divisor $D$ on $X$, the divisor
$(fg)^\natural D - g^\natural\big(f^\natural D \big)$
is effective and $g$-exceptional.
Moreover, if $\O_X(-D)\.\O_Y$ is an invertible sheaf, then
$(fg)^\natural D = g^\natural\big(f^\natural D \big)$.
\end{lem}

\begin{proof}
By Lemma~\ref{lem:natural-val:additivity:special-case}, we have
$$
(fg)^\natural (C+D) - g^\natural\big(f^\natural (C+D) \big)
= (fg)^\natural D - g^\natural\big(f^\natural D \big)
$$
for every Cartier divisor $C$. Therefore, after restricting to an open
quasi-projective subset and replacing $D$
with $C+D$ for some Cartier divisor $C \ge D$, we may assume
without loss of generality that $D$ is effective.
Then it suffices to observe that
$\O_X(-D)\.\O_Y \subseteq \O_Y(-f^\natural D)$,
with equality holding at the generic point of every
codimension one subvariety of $Y$. For the last assertion,
we first remark that the condition that $\O_X(-D)\.\O_Y$
is an invertible sheaf
remains unchanged if we multiply $\O_X(-D)$ by
an invertible sheaf $\O_X(-C)$,
and that $\O_X(-D)\.\O_Y = \O_Y(-f^\natural D)$
if $\O_X(-D)\.\O_Y$ is locally principal.
\end{proof}

By the homogeneity of valuations and pullbacks of Cartier divisors
it is very natural to extend this definition by setting
\begin{equation}\label{eq:val-pullback:Q-Cartier}
v(D) := \frac{v(mD)}{m}
\and f^*D := \frac{f^*(mD)}{m}
\end{equation}
for any $\Q$-Cartier $\Q$-divisor $D$ on $X$,
where $m$ is any non-zero integer such that $mD$ is a Cartier divisor.
In general, however, $\natural$-valuations and $\natural$-pullbacks of
arbitrary divisors do not enjoy a similar homogeneity property,
as pointed out in Example~\ref{eg:non-additivity:natural-val}.
In fact, the case of a line on a quadric cone
is an example of a $\Q$-Cartier divisor $D$ for which
$v^\natural(D) \ne v(D)$,
where the valuation on the right side is intended as defined above.
This problem is resolved by giving relevance
to the asymptotic nature of the definitions given
in \eqref{eq:val-pullback:Q-Cartier} for $\Q$-Cartier $\Q$-divisors.

\begin{lem}\label{lem:inf=limsup=lim}
For every divisor $D$ on $X$ and every $m \in \Z_+$, we have
$m\.v^\natural(D) \ge v^\natural(mD)$, and
$$
\inf_{k\ge 1} \frac{v^\natural(kD)}{k}
= \liminf_{k \to \infty} \frac{v^\natural(kD)}{k}
= \lim_{k \to \infty} \frac{v^\natural(k!D)}{k!} \in \R.
$$
\end{lem}

\begin{proof}
If $\f_1,\dots,\f_m \in \O_X(-D)(U)$ for some open set $U \subseteq X$,
then $\divisor(\f_i) \ge D$ on $U$ for each $i$, and therefore
$\divisor(\prod\f_i) \ge mD$ on $U$, which means that
$\prod \f_i \in \O_X(-mD)(U)$.
This implies that $\O_X(-D)^m \subseteq \O_X(-mD)$,
and thus $m\.v(\O_X(-D)) \ge v(\O_X(-mD))$,
which proves the first assertion of the lemma.
Both equalities in the display of the lemma follow from this
inequality, and the fact that the infimum is not $-\infty$
follows from the fact that, if $X$ is quasi-projective,
then $D \ge C$ for some Cartier divisor $C$.
In general, one reduces to the quasi-projective case by
restricting to an affine open neighborhood of the generic
point of the center of $v$ in $X$.
\end{proof}

\begin{defi}\label{defi:valuation-pullback}
Let $D$ be a $\Q$-divisor on $X$.
The {\it valuation} along $v$ of $D$ is
$$
v(D) := \lim_{k \to \infty} \frac{v^\natural(k!D)}{k!} \in \R.
$$
If $f \colon Y \to X$ is a birational morphism
from a normal variety $Y$, then the
{\it pullback} of $D$ to $Y$ is
$$
f^* D := \sum \val_E(D)\.E,
$$
where the sum is taken over all prime divisors $E$ on $Y$.
\end{defi}

\begin{prop}\label{prop:val-pullback:additivity:special-case}
Let $D$ be a $\Q$-divisor on $X$, let $v$ be a divisorial valuation
on $X$, and let $f \colon Y \to X$ be a birational morphism
from a normal variety $Y$.
If $D$ is $\Q$-Cartier, then the definition of valuation $v(D)$
and of pullback $f^*D$ given in Definition~\ref{defi:valuation-pullback}
agrees with the one given in \eqref{eq:val-pullback:Q-Cartier}.
In general, if $D$ is an arbitrary $\Q$-divisor, then
$$
v(C+D) = v(C) + v(D) \and
f^*(C+D) = f^*C + f^*D
$$
for any $\Q$-Cartier $\Q$-divisor $C$ on $X$.
\end{prop}

\begin{proof}
For clarity, we momentarily denote by $v'(D)$ the valuation
of a $\Q$-divisor $D$ as defined in Definition~\ref{defi:valuation-pullback}.
Let $C$ and $D$ be $\Q$-divisors, with $C$ $\Q$-Cartier.
We have $v^\natural(mC) = v(mC)$
for every $m \in \Z$ such that $mC$ is a Cartier divisor,
and therefore, observing that $k!C$ is a Cartier divisor
for every $k \gg 0$, we get
$$
v'(C + D)
= \lim_{k \to \infty} \frac{v^\natural(k!C+k!D)}{k!}
= \lim_{k \to \infty} \frac{v(k!C) + v^\natural(k!D)}{k!}
= v(C) + v'(D)
$$
by Proposition~\ref{lem:natural-val:additivity:special-case}.
In particular, if $D$ is $\Q$-Cartier, then we have $v'(D) = v(D)$.
The analogous statements on pullbacks follows from these.
\end{proof}

\begin{rmk}\label{rmk:comparison-natural-natural}
We have $v(D) \le v^\natural(D)$ and $f^*D \le f^\natural D$
for every divisor $D$ on $X$.
Moreover, we have $f^\natural(-D) \ge -f^\natural D$, and hence
$f^*(-D) \ge -f^*D$, for every $D$.
The following example implies that, in general, the last inequality may be strict.
\end{rmk}

The following example was found in a conversation with Lawrence Ein.

\begin{eg}\label{eg:pullback-non-additive}
Let $Y \rat Y^+$ be a flip and $f\colon Y \to X$ be the flipping
contraction, with $X$ normal and affine. Let $v$ be any divisorial valuation on $X$
whose center $C$ in $Y$ is a positive dimensional subset of
a fiber of $f$. Let $H \subset Y$
be a general hyperplane section, and let $D = f_*H$. Note that
$D$ contains $f(C)$, but $H$ does not contain $C$.
If $H^+$ is the proper transform of $H$ on $Y^+$ and $f^+ \colon Y^+ \to X$
is the contraction induced on $Y^+$, then
$$
\bigoplus_{m \ge 0} \O_X(-mD) = \bigoplus_{m\ge 0} f^+_*\O_{Y^+}(-mH^+)
$$
is finitely generated as an $\O_X$-module, since $-H^+$ is $f^+$-ample.
Therefore $v(D) = v(qD)/q$ for a sufficiently divisible $q \ge 1$.
As the ideal sheaf $\O_X(-qD)\.\O_Y$ vanishes along $C$,
it follows that $v(D) > 0$.
On the other hand, since $\O_X(mD)\.\O_Y = \O_Y(mH)$ if $m \ge 1$, we have $v(-D) = 0$,
and hence, in particular, $v(-D) \ne -v(D)$.
\end{eg}

\begin{rmk}\label{rmk:composition-natural-pullback2}
If $f\colon Y \to X$ and $g\colon V \to Y$
are birational morphisms of normal varieties, then it follows from
Lemma~\ref{prop:composition-natural-pullback}
that $(fg)^* D - g^*\big(f^* D \big)$ is effective and $g$-exceptional
for every $\Q$-divisor $D$ on $X$.
\end{rmk}

\section{Relative canonical divisors}

We recall that a {\it canonical divisor} $K_X$ on a normal variety $X$ is,
by definition, the (componentwise) closure
of any canonical divisor of the regular locus of $X$.
We also recall that $X$ is said to be
{\it $\Q$-Gorenstein} if some (equivalently, every)
canonical divisor $K_X$ is $\Q$-Cartier.

We consider a proper birational morphism
$f\colon Y \to X$ of normal varieties.
Push-forward along $f$
gives a bijection between the canonical divisors of $Y$ and
those of $X$. Moreover, if $K_Y$ and $K_Y'$ are two canonical divisors
on $Y$, then $f_*K_Y - f_*K_Y'$ is a principal
divisor and $K_Y - K_Y' = f^*(f_*K_Y - f_*K_Y')$.

Throughout the section, we fix a canonical divisor $K_Y$
on $Y$, and let $K_X = f_*K_Y$.
The standard notion of relative canonical $\Q$-divisor
(given in the case when $K_X$ is $\Q$-Cartier)
admits two generalizations
to non $\Q$-Gorenstein varieties, corresponding to whether
one pulls back $K_X$ or $-K_X$.
As we keep into consideration what the main features
of the theory of singularities of pairs are, and wish to preserve
these in our generalization, it turns out that there
are two different sides of the theory, each of which
requires a different approach, a phenomenon that
disappears in the $\Q$-Gorenstein case
(cf. Remarks~\ref{rmk:K-} and~\ref{rmk:can-v-lt}).

When dealing with the generalization of multiplier and adjoint
ideal sheaves, as well as of log canonical and log terminal singularities,
we will rely on the following notion.

\begin{defi}
For every $m \ge 1$, the {\it $m$-th limiting
relative canonical $\Q$-divisor} $K_{m,Y/X}$ of $Y$ over $X$ is
$$
K_{m,Y/X} := K_Y - \tfrac 1m\. f^\natural(mK_X).
$$
\end{defi}

On the contrary, in order to extend the definitions
of canonical and terminal singularities, we consider the following definition.

\begin{defi}
The {\it relative canonical $\R$-divisor} $K_{Y/X}$ of $Y$ over $X$ is
$$
K_{Y/X} := K_Y + f^*(-K_X).
$$
\end{defi}

Note that the definitions of $K_{m,Y/X}$
and $K_{Y/X}$ do not depend on the choice of $K_Y$.
Moreover, if $X$ is $\Q$-Gorenstein, then
$K_{Y/X}$ is the usual relative canonical $\Q$-divisor,
and it is equal to $K_{m,Y/X}$ for every $m \ge 1$ such that
$mK_X$ is Cartier.

\begin{rmk}\label{rmk:K-}
It follows by Lemma~\ref{lem:inf=limsup=lim} and
Remark~\ref{rmk:comparison-natural-natural} that
$K_{m,Y/X} \le K_{mq,Y/X} \le K_{Y/X}$ for all $m,q \ge 1$.
In particular, taking the limsup of the coefficients of the
components of the $\Q$-divisors $K_{m,Y/X}$,
one obtains the $\R$-divisor $K_{Y/X}^- := K_Y - f^*K_X$,
which satisfies $K_{Y/X}^-\le K_{Y/X}$.
Clearly the two divisors coincide if $X$ is $\Q$-Gorenstein, but
in general they may be different, as the following example shows.
\end{rmk}

\begin{eg}
With the same notation as in Example~\ref{eg:pullback-non-additive},
suppose that $-K_Y$ is $f$-ample. Then a positive multiple
$-mK_Y$ of $-K_Y$ is linearly equivalent
to a general hyperplane section $H$ of $Y$. We fix $K_X = f_*K_Y$,
and let $D = f_*H$. Note that $B := D - mK_X$ is a principal divisor,
and hence it is Cartier. Then for every birational morphism
$g \colon X' \to X$ factoring through $Y$ and extracting a divisor
with center in $Y$ equal to $C$ (cf. Example~\ref{eg:pullback-non-additive}),
we have
$$
g^*(-mK_X) = g^*B + g^*(-D) \ne
g^*B - g^*D = - g^*(mK_X)
$$
by Proposition~\ref{prop:val-pullback:additivity:special-case}
and Example~\ref{eg:pullback-non-additive}.
This implies that $g^*(-K_X) \ne - g^*K_X$,
since the pullback is by definition homogeneous
(with respect to positive multiples) on all divisors. In particular,
in this example we have $K_{X'/X}\ne K_{X'/X}^-$.
\end{eg}

\begin{lem}\label{prop:relative-K-for-composition}
Let $m$ be a positive integer, and let
$f\colon Y \to X$ be a proper
birational morphism from a normal variety $Y$
such that $mK_Y$ is Cartier and $\O_X(-mK_X)\.\O_Y$ is invertible.
Then for every proper birational morphism
$g \colon V \to Y$ from a normal variety $V$ we have
$$
K_{m,V/X} = K_{m,V/Y} + g^* K_{m,Y/X},
$$
\end{lem}

\begin{proof}
Note that $f^\natural(mK_X)$ is a Cartier divisor on $Y$, and thus
$$
K_{m,V/X} = K_V - \tfrac 1m\. (fg)^\natural(mK_X)
= K_V - \tfrac 1m\. g^\natural \big(f^\natural(mK_X)\big)
= K_V - \tfrac 1m\. g^* \big(f^\natural(mK_X)\big)
$$
by Lemma~\ref{prop:composition-natural-pullback}.
Since $mK_Y$ is Cartier, we also have
$$
K_{m,V/Y} = K_V - \tfrac 1m\. g^\natural(mK_Y) =
K_V - \tfrac 1m\. g^*(mK_Y) =  K_V - g^*K_Y,
$$
and moreover
$$
g^*K_{m,Y/X} = g^*\big(K_Y - \tfrac 1m\. f^\natural(mK_X)\big)
= g^*K_Y - \tfrac 1m\.g^*\big( f^\natural(mK_X)\big).
$$
The lemma follows.
\end{proof}

\begin{rmk}\label{rmk:composition-K}
Similarly, Remark~\ref{rmk:composition-natural-pullback2}
implies that given proper birational morphisms of normal varieties
$f \colon Y \to X$ and $g \colon V \to Y$ with $Y$ $\Q$-Gorenstein, we have
$K_{V/X} \ge K_{V/Y} + g^* K_{Y/X}$ and the difference is $g$-exceptional.
\end{rmk}

A different approach to deal with varieties that are not
$\Q$-Gorenstein, largely followed in the last decades,
is to introduce a `boundary'.
The trick is to `perturb' a canonical divisor $K_X$ of $X$
to make it $\Q$-Cartier.

\begin{defi}
An effective $\Q$-divisor $\D$ is a {\it boundary} on $X$
if $K_X + \D$ is a $\Q$-Cartier $\Q$-divisor for some
(equivalently, for any) canonical divisor $K_X$ of $X$.
If $\D$ is a boundary, then we refer to the pair
$(X,\D)$ as a {\it log variety} (or {\it variety with boundary}).
\end{defi}

\begin{defi}
Let $\D$ be a boundary on $X$, and let
$\D_Y$ be the proper transform of $\D$ on $Y$.
The {\it log relative canonical $\Q$-divisor} of $(Y,\D_Y)$ over $(X,\D)$
is given by
$$
K_{Y/X}^\D := K_Y + \D_Y - f^*(K_X + \D)
= K_Y + \D_Y + f^*(- K_X - \D).
$$
\end{defi}

If $X$ is $\Q$-Gorenstein and $\D=0$, then
$K_{Y/X}^0 = K_{Y/X} = K_{m,Y/X}$ for all sufficiently
divisible $m \ge 1$,
and $K_{Y/X}^\D$ depends on $\D$ but not on the choice of $K_X$.

\begin{rmk}\label{rmk:comparison-K-with-K(Delta)}
For every boundary $\D$ on $X$ and every $m \ge 1$
such that $m(K_X+\D)$ is Cartier, we have
$$
K_{m,Y/X} = K_{Y/X}^\D - \tfrac 1m \.f^\natural(-m\D) - \D_Y
\and
K_{Y/X} = K_{Y/X}^\D + f^*\D - \D_Y.
$$
Indeed we have $f^*(-K_X) = f^*(-K_X - \D + \D) = - f^*(K_X + \D) + f^*\D$
by Proposition~\ref{prop:val-pullback:additivity:special-case},
and similarly $f^\natural(mK_X) = f^\natural(m(K_X + \D) - m\D)
= m\.f^*(K_X + \D) + f^\natural(-m\D)$ by
Lemma~\ref{lem:natural-val:additivity:special-case}.
In particular, if $m$ is any positive integer such that
$m(K_X + \D)$ is a Cartier divisor, then $K_{Y/X}^\D \le K_{m,Y/X}$.
\end{rmk}

\section{Multiplier ideal sheaves}\label{sect:multiplier-ideals}

For the reminder of this paper, we work over an
algebraically closed field of characteristic zero.
We consider {\it pairs} of the form $(X,I)$, where
$X$ is a normal quasi-projective variety
and $I = \sum a_k\.\I_k$ is a formal $\R$-linear
combination of non-zero fractional ideal sheaves on $X$.
If each $\I_k$ is an ideal sheaf,
$Z_k \subset X$ the subscheme defined by $\I_k$,
and $Z = \sum a_k \. Z_k$ is the corresponding
formal linear combination, then
we identify the pairs $(X,I)$ and $(X,Z)$.
More generally, we allow hybrid notation
by considering pairs $(X,W + J)$, where $W$ is a formal linear combination
of proper closed subschemes and $J$ is a formal linear combination
of fractional ideal sheaves.
If $\D$ is a boundary on $X$, then
we consider {\it log pairs} of the form $((X,\D);I)$ (or more generally
of the form $((X,\D);W+J)$).

Given a formal linear combination $Z = \sum a_k\.Z_k$ on $X$,
if $f \colon Y \to X$ is a morphism such that the
scheme theoretic inverse image $f^{-1}(Z_k)$ is a Cartier divisor for
every $k$, then for short we denote $f^{-1}(Z) := \sum a_k\.f^{-1}(Z_k)$.

\begin{defi}
Consider a pair $(X,I)$ as above. A {\it log resolution} of $(X,I)$
is a proper birational morphism $f \colon Y \to X$ from
a smooth variety $Y$ such that for every $k$
the sheaf $\I_k\.\O_Y$ is the invertible sheaf of a divisor $E_k$
on $Y$, the exceptional locus $\Ex(f)$ of $f$ is also a divisor,
and $\Ex(f) \cup E$ has simple normal crossing,
where $E := \bigcup \Supp(E_k)$. If $\D$ is a boundary on $X$, then
a {\it log resolution} of the log pair $((X,\D);I)$ is given
by a log resolution $f \colon Y \to X$ of $(X,I)$ such that
$\Ex(f) \cup E \cup \Supp(f^*(K_X + \D))$ has simple normal crossings.
\end{defi}

\begin{thm}[\cite{Hir}]\label{thm:existence-of-resolution}
Let $(X,I)$ be a pair as above,
where $X$ is a normal quasi-projective
variety defined over an algebraically
closed field of characteristic zero.
Then there exists a log resolution of $(X,I)$.
If $\D$ is a boundary on $X$, then
there exists a log resolution of singularities of $((X,\D);I)$.
\end{thm}

\begin{proof}
Let $C$ be a Cartier divisor on $X$ such that
$\O_X(C)\.\I_k \subseteq \O_X$ for every $k$,
and let $W = \sum a_k\.W_k$, where $W_k \subset X$ is the
subscheme defined by the ideal sheaf $\O_X(C)\.\I_k$.
Then any log resolution of the pair $(X,W + \Supp(C))$
(respectively, of $((X,\D);W + \Supp(C))$)
is a log resolution of $(X,I)$ (respectively, of $((X,\D);I)$).
This reduces the theorem to the original version due to Hironaka.
\end{proof}

\begin{defi}
We say that $Z$, or $(X,Z)$, is {\it effective} if $a_k \ge 0$
for all $k$. If $\D$ is a boundary on $X$, then
we say that the log pair $((X,\D);Z)$ is {\it effective} if so is $Z$.
\end{defi}

Consider now an arbitrary effective pair $(X,Z)$.
Because of the possible failure of functoriality
for composition of pullback of arbitrary $\Q$-divisors
(cf. Lemma~\ref{prop:composition-natural-pullback}),
the definition of multiplier ideal sheaf of $(X,Z)$ requires
some preparation.

We fix a canonical divisor $K_X$ on $X$.
For any fixed integer $m \ge 1$, we consider
a log resolution $f \colon Y \to X$ of the pair $(X,Z + \O_X(-mK_X))$,
and define
$$
\J_m(X,Z):= f_*\O_Y(\lru K_{m,Y/X} - f^{-1}(Z)\rru).
$$
When $Z = 0$, we denote this sheaf by $\J_m(X)$.

The proof of the next proposition is similar to the proof of the
analogous properties for multiplier ideals in the $\Q$-Gorenstein case;
we give it for completeness.

\begin{prop}\label{prop:J_m}
The sheaf $\J_m(X,Z)$ is a (coherent) sheaf of
ideals on $X$, and its definition
is independent of the choice of $f$.
\end{prop}

We start with two lemmas.

\begin{lem}\label{lem:Fujita}
Let $f \colon Y \to X$ be a proper birational morphism
from a smooth variety $Y$ to a normal variety $X$,
let $P$ and $N$ be effective divisors on $Y$
without common components, and suppose that $P$ is $f$-exceptional.
Then $f_*\O_Y(P-N) = f_*\O_Y(-N)$.
\end{lem}

\begin{proof}
By a lemma of Fujita, which gives the vanishing $f_*\O_P(P) = 0$
(see \cite[Lemma~1-3-2]{KMM}), and hence $f_*\O_P(P-N) = 0$.
\end{proof}

\begin{lem}\label{lem:EV-from-Laz}
Let $g \colon Y' \to Y$ be a proper birational morphism of smooth
varieties, and let $D$ be an effective $\R$-divisor on $Y$ with simple normal
crossing support.
Then
$$
g_*\O_{Y'}\big(K_{Y'/Y} + \lru g^*D \rru \big) = \O_Y\big(\lru D \rru \big).
$$
\end{lem}

\begin{proof}
See \cite[Lemma~9.2.19 and Remark~9.2.10]{Laz}.
\end{proof}

\begin{proof}[Proof of Proposition~\ref{prop:J_m}]
Let $f'\colon Y' \to X$ be another log-resolution of $(X,Z + \O_X(-mK_X))$.
Since we can always compare $f$ and $f'$ with a common resolution,
we may assume, without loss of generality, that $f'$ factors through
$f$ and a morphism $g\colon Y' \to Y$.
By Lemma~\ref{prop:relative-K-for-composition},
we have $K_{m,Y'/X} = K_{Y'/Y} + g^*K_{m,Y/X}$, and hence
$$
(fg)_*\O_{Y'}\big(\lru K_{m,Y'/X} - (fg)^{-1}(Z) \rru\big)
= {f}_*\big(g_*\O_{Y'}\big(K_{Y'/Y} +
\lru g^*(K_{m,Y/X} - f^{-1}(Z)) \rru\big)\big).
$$
Therefore, by Lemma~\ref{lem:EV-from-Laz}, we obtain
$$
(fg)_*\O_{Y'}\big(\lru K_{m,Y'/X} - (fg)^{-1}(Z) \rru\big)
= {f}_*\O_{Y}\big(\lru K_{m,Y/X} - f^{-1}(Z) \rru\big).
$$
This proves the independence of the definition from the choice of $f$.
The fact that $\J_m(X,Z)$ is a sheaf of ideals follows from
Lemma~\ref{lem:Fujita}.
\end{proof}

\begin{prop}\label{prop:set-J_m-has-unique-max-element}
The set of ideal sheaves $\{\J_m(X,Z)\}_{m \ge 1}$
has a unique maximal element.
\end{prop}

\begin{proof}
We have $\J_m(X,Z) \subseteq \J_{mq}(X,Z)$ for all $m,q \ge 1$.
Indeed, by Proposition~\ref{prop:J_m},
this inclusion can be verified for any choice of $m$
and $q$ by taking a log resolution $f \colon Y \to X$
of $(X,Z + \O_X(-mK_X) + \O_X(-mqK_X))$ and applying the first formula in
Remark~\ref{rmk:K-}.
Therefore, by the Noetherian property of $X$,
the set of ideal sheaves $\{\J_m(X,Z)\}_{m \ge 1}$
has a unique maximal element.
\end{proof}

\begin{defi}
Let $(X,Z)$ be an effective pair.
The unique maximal element of $\{\J_m(X,Z)\}_{m \ge 1}$
is called the {\it multiplier ideal sheaf} of $(X,Z)$,
and is denoted by $\J(X,Z)$. If $Z$ is trivial, then
we denote the corresponding multiplier ideal
sheaf by $\J(X)$.
\end{defi}

Note that $\J(X,Z) = \J_m(X,Z)$ for all sufficiently divisible $m \ge 1$.
If $X$ is $\Q$-Gorenstein, then
this definition of multiplier ideal sheaf agrees with the usual one.

We close this section with some basic properties of multiplier ideals.

\begin{prop}\label{prop:basic-properties-mult-ideal}
Let $Z = \sum b_k\.Z_k$ be an effective linear combination
of proper closed subschemes of a normal variety $X$.
\begin{enumerate}
\item
If $Z' = \sum b_k'\.Z_k'$ with $b_k' \ge b_k$ and $\I_{Z_k'} \subseteq \I_{Z_k}$
for all $k$, then $\J(X,Z') \subseteq \J(X,Z)$.
\item
There is an $\ep > 0$ such that
$\J(X,(1+t)Z) = \J(X,Z)$ for all $0 < t \le \ep$.
\end{enumerate}
\end{prop}

\begin{proof}
We can fix $m$ such that $\J(X,Z) = \J_m(X,Z)$ and $\J(X,Z') = \J_m(X,Z')$.
The first property is then immediate from the definition of these
ideal sheaves. Regarding~(b), we observe that
$\J_m(X,Z) = \J_m(X,(1+t)Z)$ for all $0 < t \ll 1$. Thus the property
follows by the chain of inclusions
$$
\J_m(X,(1+t)Z) \subseteq \J(X,(1+t)Z) \subseteq \J(X,Z) = \J_m(X,Z),
$$
the second of which holding by part~(a).
\end{proof}

\begin{rmk}
One can define the jumping numbers of an effective pair $(X,Z)$
in a similar fashion as in the $\Q$-Gorenstein case, by declaring
that a number $\m > 0$ is a {\it jumping number} of an effective pair $(X,Z)$
if $\J(X,\l Z) \ne \J(X,\m Z)$ for all $0 \le \l < \m$.
It would be interesting to study the properties of these numbers.
For instance, is the set of jumping numbers of an effective pair
a discrete set of rational numbers?
\end{rmk}

\section{First properties and applications}

As we will see below, it turns out that multiplier ideals
(as defined in the previous section) can actually be
realized as multiplier ideals of suitable log pairs.
Using this fact, we will see that
the main features of the theory automatically extend to our setting.

\begin{defi}
Let $(X,Z)$ be an effective pair, and fix an integer $m \ge 2$.
Given a log resolution $f \colon Y \to X$ of $(X,Z+\O_X(-mK_X))$,
a boundary $\D$ on $X$ is said to be
{\it $m$-compatible} for $(X,Z)$ with respect to $f$ if:
\begin{enumerate}
\item[(i)]
$m\D$ is integral and $\lrd \D \rrd = 0$,
\item[(ii)]
no component of $\D$ is contained in the support of $Z$,
\item[(iii)]
$f$ is a log resolution for the log pair $((X,\D);Z+\O_X(-mK_X))$, and
\item[(iv)]
$K_{Y/X}^\D = K_{m,Y/X}$.
\end{enumerate}
The pair $(X,Z)$ is said to {\it admit $m$-compatible boundaries}
if there are $m$-compatible boundaries with respect
to any sufficiently high log resolution of $(X,\O_X(-mK_X) + Z)$.
\end{defi}

This definition is motivated by the following useful property.

\begin{prop}\label{prop:mult-ideal-for-compatible-boundaries}
If $(X,Z)$ is an effective pair and $m \ge 2$ is such that
$\J(X,Z) = \J_m(X,Z)$, then we have
$$
\J(X,Z) = \J((X,\D);Z)
$$
for any $m$-compatible boundary $\D$, where $\J((X,\D);Z)$ is the
multiplier ideal sheaf of the log pair $((X,\D);Z)$ as defined in
\cite[Definition~9.3.56]{Laz}.
\end{prop}

\begin{proof}
It suffices to observe that if $\D$ has no common components with $Z$
and $\lrd \D \rrd = 0$, then
$$
\J((X,\D);Z) = f_*\O_Y(\lru K_Y - g^*(K_X + \D) - f^{-1}(Z) \rru )
= f_*\O_Y(\lru K^\D_{Y/X} - f^{-1}(Z) \rru )
$$
for any log resolution $f$ of $((X,\D);Z)$.
\end{proof}

\begin{rmk}\label{rmk:comparison-mult-ideals}
In general, if $(X,Z)$ is an effective pair
and $\D$ is a boundary on $X$, then $\J((X,\D);Z) \subseteq \J(X,Z)$.
Indeed, if $m \ge 1$ such that $m(K_X + \D)$ is a Cartier divisor and
$\J(X,Z) = \J_m(X,Z)$, then the inclusion follows
from the last formula in Remark~\ref{rmk:comparison-K-with-K(Delta)}.
\end{rmk}

\begin{thm}\label{thm:existence-of-compatible-boundary}
Every effective pair $(X,Z)$ admits $m$-compatible boundaries
for any $m \ge 2$.
\end{thm}

\begin{proof}
Let $D$ be an effective divisor such that $K_X - D$ is Cartier,
and let $f \colon Y \to X$ be a log resolution of
$(X,\O_X(-mK_X) + \O_X(-mD))$, and let $E = f^\natural(mD)$, so that
$\O_X(-mD)\.\O_Y = \O_Y(-E)$.
Since
$$
f^\natural(mK_X) = f^\natural(m(K_X - D)+mD) = m\.f^*(K_X - D) + f^\natural(mD),
$$
we have
$$
K_{m,Y/X} = K_Y - f^*(K_X - D) - \tfrac 1m E.
$$
Let $\LL$ be an invertible sheaf on $X$ such that $\LL \otimes \O_X(-mD)$
is globally generated, and let $G$ be a general element in the
linear system $\{L \in |\LL| \mid L - mD \ge 0 \}$.
Then $G = M + mD$ and $f^*G = M_Y + E$, where $M$ is an effective divisor
and $M_Y$ is its proper transform.
As $G$ varies, $M_Y$ moves in a base point free linear system.
In particular, we can assume that $M$
is a reduced divisor with no common components with $D$ or $Z$. We let
$$
\D := \tfrac 1m M.
$$
Note that $m\D$ is integral, $\lrd \D \rrd = 0$,
and $K_X + \D = K_X - D + \tfrac 1m G$ is $\Q$-Cartier.
Moreover, by choosing $G$ general, we can also assume that
$f$ is a log resolution for $((X,\D);\O_X(-mK_X))$.
The fact that $\D$ is $m$-compatible follows then by the computation
\begin{align*}
K^\D_{Y/X} &= K_Y + \D_Y - f^*(K_X + \D) \\
&= K_Y + \D_Y - f^*(K_X + \D - \tfrac 1m G ) - \tfrac 1m f^*G \\
&= K_Y  - f^*(K_X - D) - \tfrac 1m E.
\end{align*}
\end{proof}

We deduce the following fact. A posteriori, one can use this
corollary as the definition of $\J(X,Z)$.

\begin{cor}
For any effective pair $(X,Z)$, the set of ideal sheaves
$$
\{\J((X,\D);Z) \mid \text{$\D$ is a boundary on $X$}\}
$$
has a unique maximal element, namely $\J(X,Z)$.
\end{cor}

Using $m$-compatible boundaries, we obtain the following
generic restriction result.

\begin{prop}
Let $(X,Z)$ be an effective pair. If $H \subset X$ is a
general hyperplane section, then
$\J(X,Z) \. \O_H = \J(H,Z|_H)$.
\end{prop}

\begin{proof}
If $\D$ is a boundary on $X$ with no common components with $H$, then
$\D|_H$ is a boundary on $H$, and
we have $\J((X,\D),Z) \. \O_H = \J((H,\D|_H),Z|_H)$
(cf. \cite[Example~9.5.9]{Laz}). Suppose that $\D$ is
a $m$-compatible boundary on $X$ for some $m$ sufficiently divisible.
By Remark~\ref{rmk:comparison-mult-ideals} applied on $H$, we see immediately
that $\J(X,Z) \. \O_H \subseteq \J(H,Z|_H)$. To get an equality,
we need to show that the restriction $\D|_H$ of $\D$ to $H$ is
also $m$-compatible, if $H$ is sufficiently general.

To this end, we fix a canonical divisor $K_0$ on $X$.
Working locally on $X$, we may assume that $K_0$ is effective.
Assume that $H$ is general with respect to
$Z$, $\D$ and $K_0$. Then we replace $K_0$ by $K_X := K_0 - H + H_0$,
where $H_0$ is another general hyperplane section linearly equivalent to $H$.
Note that $K_H := (K_X + H)|_H$ is a canonical divisor on $X$.

We claim that
\begin{equation}\label{eq:last:claim}
\O_X(-m(K_X + H)) \.\O_H = \O_H(-mK_H) \fall m \ge 0.
\end{equation}
For short, let $B := m(K_X + H) = mK_0 + mH_0$.
Note that $H$ has been chosen generally
with respect to $B$. Let $g \colon X' \to X$ be a resolution of singularities
of $X$. Let $B'$ be the proper transform of $B$.
We can assume that the pullback of $H$ to $X'$ coincides with
its proper transform $H'$, and moreover that $B'|_{H'}$
is the proper transform of $B|_H$. Note also that,
since $B$ is effective,
$g_*\O_{X'}(-B') = \O_X(-B)$, and similarly, $g_*\O_{H'}(-B'|_{H'}) = \O_H(-B|_H)$.
On $X'$ we have the exact sequence
$$
0 \to \O_{X'}(-B') \otimes g^*\O_X(-H) \to \O_{X'}(-B') \to \O_{H'}(-B'|_{H'}) \to 0.
$$
Since $H$ is generic, the map
$R^1g_*\O_{X'}(-B') \otimes \O_X(-H) \to R^1g_*\O_{X'}(-B')$
is injective. Therefore, taking direct images, we obtain a surjection
$\O_{X}(-B) \to \O_{H}(-B|_{H})$,
which shows that \eqref{eq:last:claim} holds.

We take a log resolution $f \colon Y \to X$ of $(X,Z + \O_X(-mK_X) + H)$.
Let $\~H$ be the proper transform of $H$. Since $\D$ is a $m$-compatible
boundary, we can assume that $K_{Y/X}^\D = K_{m,Y/X}$.
On the other hand, using \eqref{eq:last:claim} we see that
$$
K_{m,Y/X}|_{\~H} = K_{m,\~H/H}.
$$
Since adjunction in the log $\Q$-Gorenstein case implies that
$$
K_{Y/X}^\D|_{\~H} = K_{\~H/H}^{\D|_H},
$$
we conclude that $\D|_H$ is a $m$-compatible boundary for $(H,Z|_H)$
(the other defining conditions of $m$-compatible being easily verified).
This concludes the proof of the proposition.
\end{proof}

The existence of $m$-compatible boundaries allows us to
deduce immediately many other properties of multiplier ideal sheaves.
We start with Skoda's theorem, which extends to our setting
in a straightforward manner.

\begin{cor}
If $\a \subseteq \O_X$ be a non-zero ideal sheaf on an $n$-dimensional
normal variety $X$, then for every integer $m \ge n$
$$
\J(X,\a^m) = \a^{m+1-n}\.\J(X,\a^{n-1}).
$$
\end{cor}

\begin{proof}
After fixing an $m$-compatible boundary, the result follows from
\cite[Variation~9.6.39]{Laz}.
\end{proof}

The main application, however, is the
following extension of Nadel's vanishing theorem
(\cite{EV1,Nad1,Nad2}; see also \cite[Section~9.4]{Laz}).

\begin{cor}\label{thm:vanishing}
Let $(X,Z)$ be an effective pair, where $X$ is a projective
normal variety and $Z = \sum a_k\.Z_k$.
Let $m \ge 2$ be an integer such that $\J(X,Z)=\J _m(X,Z)$,
and let $\D$ be an $m$-compatible boundary for $(X,Z)$.
For each $k$, let $B_k$ be a Cartier divisor
such that $\O_X(B_k)\otimes \I_{Z_k}$ is
globally generated, and suppose that $L$ is a Cartier divisor such that
$L-\big(K_X + \D + \sum a_kB_k\big)$ is nef and big. Then
$$
H^i\big(\O_X(L) \otimes \J(X,Z)\big) = 0 \for i > 0.
$$
\end{cor}

\begin{proof}
It follows by Proposition~\ref{prop:mult-ideal-for-compatible-boundaries}
and \cite[Theorem~9.4.17]{Laz}.
\end{proof}

As in the log $\Q$-Gorenstein case
(cf. \cite[Section~9.4.E]{Laz}), one obtains the following.

\begin{cor}
With the same notation and assumptions as in
Corollary~\ref{thm:vanishing}, let $A$ be a
very ample Cartier divisor on $X$. Then the sheaf $\O_X(L + kA) \otimes \J(X,Z)$
is globally generated for every integer $k \ge \dim X +1$.
\end{cor}

The existence of $m$-compatible boundaries
also implies the following relative vanishing.

\begin{cor}\label{cor:relative-van}
Let $X$ be a normal quasi-projective variety.
Then for any integer $m \ge 2$ and every sufficiently high
log resolution $f\colon Y\to X$ of the pair $(X,\O_X(-mK_X))$ we have
$$
R^if_*\O_Y\big(\lru K_{m,Y/X}  \rru\big)=0 \for i > 0.
$$
\end{cor}

We close the section with a generalization of
Shokurov--Koll\'ar's connectedness lemma \cite{Ko0,Sho}.

\begin{cor}\label{cor:conn}
With the same notation and assumptions as in
Corollary~\ref{thm:vanishing}, let
$$
K_{m,Y/X}-f^{-1}(Z)=\sum e_iE_i=A-B
\ \ \text{where} \ A=\sum_{e_i>-1}e_iE_i.
$$
Assume that $\lru A\rru$ is exceptional,
(i.e., that all divisorial components of $Z$ appear with coefficient less than $1$).
Then $\Supp(B)$ is connected in a neighborhood of any fiber of $f$.
If moreover $B$ is irreducible and reduced, then $f(B)$ is normal.
\end{cor}

\begin{proof}
By Theorem~\ref{thm:existence-of-compatible-boundary} and
Proposition~\ref{prop:mult-ideal-for-compatible-boundaries},
we reduce to the log $\Q$-Gorenstein case,
where the result is well known (cf. \cite[Theorem~7.4]{Kol} or
\cite[Theorem~1.6]{Kaw4}).
\end{proof}

\section{Asymptotic constructions and adjoint ideal sheaves}\label{sect:adj-ideals}

This section is devoted to a discussion of
asymptotic multiplier ideal sheaves and adjoint ideal sheaves.
As in the last two sections, we work over an algebraically
closed field of characteristic zero.

Let $D$ be a divisor on a normal variety $X$, and for every $n \ge 1$
let $B_n \subset X$ denote the base scheme of the
linear system $|nD|$. We suppose that $|n_0D| \ne \emptyset$
for some $n_0  \ge 1$, and let $N = n_0\.\Z_+$. As in the usual case
(cf. \cite[Chapter~11]{Laz}),
for any given $c > 0$ the set of multiplier ideal sheaves
$\{\J(X,\frac cn\.B_n)\}_{n \in N}$
has a unique maximal element (which does not depend on the
choice of $n_0$).

\begin{defi}
The unique maximal element of $\{\J(X,\frac cn\.B_n)\}_{n \in N}$
is denoted by $\J(X,c\.\| D\|)$, and is called the
{\it asymptotic multiplier ideal sheaf} of $D$ with weight $c$.
\end{defi}

Note that if $m$ is sufficiently divisible and $\D$ is an $m$-compatible boundary
for $(X,B_{n_0})$, then $\J(X,\| D \|) = \J((X,\D);\| D \|)$.
We deduce the following property (cf. \cite[Theorem~11.1.8 and Remark~11.1.13]{Laz}).

\begin{prop}
With the above notation, suppose that $\J(X) = \O_X$
(i.e., $X$ is `log terminal', see
Definition~\ref{defi:lc-lt-sings} below).
Then $\J(X,\| D\|)$ contains the ideal sheaf of the base
scheme of $|D|$. In particular, we obtain
$$
H^0\big(\O_X(D)\.\J(X,\| D\|)\big) \cong H^0\big(\O_X(D)\big).
$$
\end{prop}

We next define the adjoint ideal sheaf of
an effective pair $(X,Z)$ along an effective Cartier divisor $H$.
We fix a log resolution $f \colon Y \to X$ of
$(X,Z + \O_X(-mK_X))$ such that all
components of the proper transform $H_Y$ of $H$ on $Y$ are disconnected;
if $\D$ is a given boundary on $X$, then we also
suppose that $f$ is a log resolution of the log pair $((X,\D);Z + \O_X(-mH))$.
Then we consider the ideal sheaf
$$
\adj_{m,H}(X,Z)
:= f_*\O_Y\big(\lru K_{m,Y/X} - f^{-1}(Z) - f^*H + H_Y\rru\big).
$$
Again, one can check that
$\adj_{m,H}(X,Z)$ is a (coherent) sheaf of ideals on $X$, that its definition
is independent of the choice of $f$
and that the set of ideal sheaves $\{\adj_{m,H}(X,Z)\}_{m \ge 1}$
has a unique maximal element.

\begin{defi}
The maximal element of $\{\adj_{m,H}(X,Z)\}_{m \ge 1}$ is called the
{\it adjoint ideal sheaf} of the pair $(X,Z)$ along $H$,
and is denoted by $\adj_H(X,Z)$.
\end{defi}

\begin{rmk}
If $\D$ is an $m$-compatible boundary for some $m$ sufficiently divisible,
then $\adj_H(X,Z) = \adj_H((X,\D);Z)$.
\end{rmk}

\begin{prop}\label{prop:adjoint-exact-seq}
Suppose that $H$ is a normal Cartier divisor on $X$
with no components contained in the support of $Z$, and let
$\D$ be an $m$-compatible boundary for $(X,Z + H)$ for a
sufficiently divisible $m$.
Then the adjoint ideal $\adj_H(X,Z)$ sits in the exact sequence
$$
0 \to \J(X,Z)\otimes \O_X(-H) \to \adj_H(X,Z) \to \J((H,\D|_H);Z|_H) \to 0.
$$
\end{prop}

\begin{proof}
If $\D$ is an $m$-compatible boundary for $(X,Z + H)$ for a
sufficiently divisible $m$, then $\J(X,Z) = \J((X,\D);Z)$ and
$\adj_H(X,Z) = \adj_H((X,\D);Z)$. Therefore the result
follows from the log $\Q$-Gorenstein case, in which case it is well known
(see, for example, the arguments in \cite[Proposition~2.4]{Tak}).
\end{proof}

\begin{rmk}
One can try to apply adjunction directly, without adding the boundary divisor,
by fixing a canonical divisor $K_X$ on $X$ such that $K_X + H$
had order zero along the components of $H$.
Then $K_H := (K_X + H)|_H$ is a canonical divisor on $H$
(cf. Remark~5.47 in \cite{KM}).
However, $\O_X(-m(K_X+H))\.\O_H$ may in general be strictly
contained in $\O_H(-mK_H)$ if $K_X + H$ is not $\Q$-Cartier.
This reflects the fact that in general, no matter how one chooses
$\D$ on $X$, $\J((H,\D|_H);Z|_H)$ may be strictly smaller than $\J(H,Z|_H)$,
as it happens in the following example.
\end{rmk}

\begin{eg}
As in \cite[Example~4.3]{Kaw3},
we consider an extremal flipping contraction $\phi \colon X' \to X$
on a normal $\Q$-factorial threefold $X'$ with terminal singularities.
We assume that $X$ is affine, and let $0 \in X$ be the image of the
exceptional locus of $\f$. Let $H \subset X$ be a general hyperplane section
through $0$, and let $H' = f^{-1}(H) \subset X'$.
We assume that $H$ and $H'$ are normal $\Q$-factorial surfaces
with log terminal singularities (this is the case, for instance,
if $\phi$ is the contraction in Francia's flip \cite{Fra}). Note that
$\f$ restricts to a divisorial contraction $\ff \colon H'  \to H$.
Let $C$ be an irreducible component of $\Ex(\ff)$.
Let $\D$ be any boundary on $X$ not containing $H$ in its support,
and let $\D'$ be its proper transform on $X'$.
Then $\D' \. C = -K_{X'}\.C > 0$.
It follows that $\val_C(\D|_H)$
is positive and independent of the choice of $\D$.
This implies that there is a $\d > 0$, independent of $\D$, such that
if $Z = \{0\} \subset H$, then
$\lc((H,\D|_H);Z) \le \lc(H,Z) - \d$.
Therefore we can fix $c > 0$, independent of $\D$, such that
$\J((H,\D|_H);cZ) \ne \O_H$ but $\J(H,cZ)=\O_H$.
\end{eg}

We immediately obtain the following inversion of adjunction statement.

\begin{cor}\label{cor:inversion-of-adj}
In the same assumptions as in Proposition~\ref{prop:adjoint-exact-seq}, we have
$\adj_H(X,Z)=\O_X$ in a neighborhood of $H$
if and only if $\J((H,\D|_H);Z|_H)=\O _H$.
\end{cor}

\begin{rmk}\label{rmk:KM-Thm5.50}
The corollary above should be compared with the following
well known statement: If $S\subset X$ is a normal Cartier divisor
(in fact, it suffices to assume that $S$ is
Cartier in codimension $2$) on a
normal variety and $B$ is an effective divisor such that $K_X+S+B$ is
$\Q$-Cartier, then $(X,S+B)$ is purely log terminal
with respect to $S$ if and only if $(S,B|_S)$ is Kawamata log terminal
(see for example \cite[Theorem~5.50]{KM}).
\end{rmk}

We also obtain the following vanishing theorem for adjoint ideals.

\begin{cor}\label{cor:vanishing-for-adj-ideals}
With the same assumptions as in Corollary~\ref{thm:vanishing},
let $H$ be a general hyperplane section of $X$. Then
$$
H^i\big(\O_X(L + H) \otimes \adj_H(X,Z)\big) = 0 \for i > 0.
$$
\end{cor}

\begin{proof}
We have
$$
H^i\big(\O_X(L) \otimes \J(X,Z)\big) = 0 \for i > 0
$$
by Corollary~\ref{thm:vanishing}.
Observe that the restriction $\D|_H$ of $\D$ to $H$ is
a boundary on $H$. Moreover, the sheaves
$\O_H(B_k|_H) \otimes \I_{Z_k|_H}$ are globally generated, and
$(L-(K_X+\Delta + \sum a_kB_k))|_H$ is nef and big.
By adjunction, this implies that
$$
(L + H)|_H - \big(K_H + \D|_H + \sum a_kB_k|_H\big)
$$
is nef and big, and hence we have
$$
H^i\big(\O_H(L+H) \otimes \J((H,\D|_H),Z|_H)\big) = 0 \for i > 0.
$$
The assertion then follows by Proposition~\ref{prop:adjoint-exact-seq}.
\end{proof}

\begin{rmk}
In fact it suffices to assume that $H$ is a normal
Cartier divisor that is not contained in the augmented base locus
of $L-(K_X+\Delta + \sum a_kB_k)$.
\end{rmk}

\section{Log terminal and log canonical singularities}

In this section we extend the definitions of log terminal
and log canonical singularities of pairs to the general setting,
and discuss some generalizations to this context of
certain results on rational and log terminal
singularities due, respectively, to Elkik and Kawamata.

Let $(X,Z)$ be an effective pair over
and algebraically closed field of characteristic zero.

\begin{defi}\label{defi:lc-lt-sings}
Let $X' \to X$ be a proper birational morphism with $X'$ normal,
and let $F$ be a prime divisor on $X'$.
For any integer $m \ge 1$, we define the
{\it $m$-th limiting log discrepancy} of $(X,Z)$ to be
$$
a_{m,F}(X,Z) := \ord_F(K_{m,X'/X}) + 1 - \val_F(Z).
$$
The pair $(X,Z)$ is said to be
{\it log canonical} (resp., {\it log terminal}) if
there is an integer $m_0$ such that
$a_{m,F}(X,Z) \ge 0$ (resp., $> 0$) for every prime divisor $F$ over $X$
and $m=m_0$ (and hence for any positive multiple $m$ of $m_0$).
$(X,Z)$ is said to be {\it strictly log canonical} if it is log canonical
but not log terminal.
If $X$ is log terminal, then the {\it log canonical threshold}
of $(X,Z)$ is
$$
\lc(X,Z) := \sup \{ t > 0 \mid \text{$(X,tZ)$ is log terminal} \}.
$$
\end{defi}

Clearly these notions coincide with the usual ones when $X$ is $\Q$-Gorenstein,
and in general, if $(X,Z)$ is log terminal, then it is log canonical.

\begin{prop}\label{prop:lt-lc-reduction-to-log-pairs}
An effective pair $(X,Z)$ is log canonical (resp., log terminal)
if and only if there is a boundary $\D$
such that $((X,\D); Z)$ is log canonical (resp., log terminal).
\end{prop}

\begin{proof}
If there is a boundary $\D$
such that $((X,\D); Z)$ is log canonical (resp., log terminal),
then it follows by Remark~\ref{rmk:comparison-K-with-K(Delta)}
that $(X,Z)$ is log canonical (resp., log terminal).
Conversely, assume that $(X,Z)$ is log canonical (resp., log terminal), and
let $m_0$ be as in Definition~\ref{defi:lc-lt-sings}.
By Theorem~\ref{thm:existence-of-compatible-boundary}, there is
an $m_0$-compatible boundary $\D$ for $(X,Z)$. Given any prime divisor $F$
over $X$, we can assume that $F$ is a divisor over a sufficiently high
log resolution $Y$ of $(X,Z + \O_X(-m_0K_X))$.
Then $K^\D_{Y/X} = K_{m_0,Y/X}$, and hence
$a_{m_0}((X,\D);Z) = a_{m_0}(X,Z)$. It follows that $((X,\D);Z)$
is log canonical (resp., log terminal).
\end{proof}

The next corollary shows the
relation between our notion of log canonical singularities and Nakayama's
notion of {\it admissible singularities}
(see \cite[Defintion~VII.1.2]{Nak}); we are grateful to Hara, Schwede and Takagi
for bringing Nakayama's notion to our attention.

\begin{cor}
An effective pair $(X,Z)$ is log terminal if and only if
it has admissible singularities in the sense of Nakayama.
\end{cor}

\begin{proof}
If follows by comparing \cite[Lemma~VII.1.3]{Nak} with
Proposition~\ref{prop:lt-lc-reduction-to-log-pairs} and the fact that
our notion is local.
\end{proof}

\begin{rmk}\label{rmk:comparing-log-sings}
In general, taking an arbitrary boundary $\D$,
if $((X,\D);Z)$ is log terminal (resp., log canonical),
then so is $(X,Z)$. In particular, if $(X,\D)$ is log terminal, then
$\lc((X,\D);Z) \le \lc(X,Z)$.
\end{rmk}

\begin{cor}
Let $(X,Z)$ be a log canonical (resp., log terminal) effective pair.
If $H \subset X$ is a general hyperplane section, then
$(H,Z|_H)$ is log canonical (resp., log terminal).
\end{cor}

\begin{proof}
Since $((X,\D);Z)$ is log canonical (resp., log terminal) for some boundary $\D$,
so is $((H,\D|_H);Z|_H)$, and hence $(H,Z|_H)$.
\end{proof}

\begin{cor}\label{prop:characterization-sings-general-case}
An effective pair $(X,Z)$ is log terminal if and only if
$\J(X,Z) = \O_X$. Moreover, if $X$ is log terminal, then
$$
\lc(X,Z) = \sup \{ t > 0 \mid \J(X,tZ) = \O_X \}.
$$
\end{cor}

We next address the extension of
Elkik's theorem on rational singularities \cite{Elk}.

\begin{cor}\label{thm:lt-implies-rational}
Let $X$ be a normal variety with log terminal singularities. Then
$X$ has rational singularities.
\end{cor}

\begin{proof}
The proof follows from \cite[Theorem 5.22]{KM}
as there exists a boundary $\D$ such that $\J(X,\D) = \O_X$.
\end{proof}

Similarly, the analogous result on Du Bois singularities
due to Kov\'acs, Schwede and Smith \cite[Theorem~1.2]{KSS}
generalizes as follows.

\begin{cor}
Let $X$ be a normal Cohen--Macaulay variety with log canonical singularities. Then
$X$ has Du Bois singularities.
\end{cor}

We also obtain the following generalization of \cite[Theorem~5.5]{Sch},
which was kindly brought to our attention by Karl Schwede.

\begin{cor}
Let $(X,Z)$ be an effective pair with log canonical singularities,
and suppose that $X$ is log terminal. Then the multiplier ideal
$\J(X,Z)$ defines a scheme with Du Bois singularities.
\end{cor}

In \cite{Kaw}, Kawamata proves an important result on
the singularities of minimal log canonical centers,
which in particular implies that such centers have rational singularities.
It follows immediately that,
in the setting and terminology of Theorem~1 in \cite{Kaw},
`minimal log canonical centers' are normal varieties
with log terminal singularities.
In particular this appears to be a natural setting for the theory developed in this
paper, as in general `minimal log canonical centers' are not known to be
$\Q$-Gorenstein (even when the ambient variety is smooth).

In fact, Kawamata's subadjunction theorem extends to our general setting.

\begin{defi}
Let $(X,Z)$ be an effective strictly log canonical pair,
and let $m_0$ be as in Definition~\ref{defi:lc-lt-sings}.
A subvariety $W \subset X$ is said to be
a {\it log canonical center} of $(X,Z)$
if for every multiple $m$ of $m_0$
there is a exceptional prime divisor $E$ over
$X$ such that $c_X(E) = W$ and $a_{m,E}(X,Z) = 0$.
A log canonical center is said to be {\it minimal}
if it is so with respect to inclusions.
\end{defi}

\begin{prop}\label{prop:log-can-centres}
Let $W \subseteq X$ is a minimal log canonical
center for an effective strictly log canonical pair $(X,Z)$.
Then for any sufficiently divisible $m$, there is an effective $m$-compatible
boundary $\D$ such that $W$ is a minimal log canonical center for $((X,\D);Z)$.
\end{prop}

\begin{proof}
Let $m_0$ as in Definition~\ref{defi:lc-lt-sings}, and for every integer
$k > 0$, let $\D_k$ be a $km_0$-compatible boundary for $(X,Z)$.
Note that, for every $k$, the pair $((X,\D_k);Z)$ is log canonical
and $W$ is a log canonical center for $((X,\D_k);Z)$.
Moreover, for every $k \ge n \ge 1$ we have
$a_F((X,\D_k);Z) \ge a_F((X,\D_n);Z)$ for any divisor $F$ over $X$.
It follows that, if $\W_k$ denotes the set
of log canonical centers of $((X,\D_k);Z)$,
then $\W_k \subseteq \W_n$ for every $k \ge n \ge 1$.
Since a strictly log canonical log pair has only finitely many
log canonical centers, the sequence of sets $\{\W_k\}$ stabilizes,
and therefore $W$ is a minimal log canonical center of $((X,\D_k);Z)$,
for $k \gg 1$.
\end{proof}

\begin{cor}\label{thm:subadj-general-case}
Let $(X,Z)$ be an effective strictly log canonical pair on a
log terminal variety $X$.
Then every minimal log canonical center of $(X,Z)$ is a normal
variety with log terminal (and hence rational) singularities.
\end{cor}

\begin{proof}
Let $W$ be a minimal log canonical center.
By Proposition~\ref{prop:log-can-centres}, we can fix an
$m$-compatible boundary $\D$ such that $W$ is a minimal log canonical center
of $((X,\D),Z)$. It follows by \cite{Kaw} that there is a
boundary $\D_W$ on $W$ such that $(W,\D_W)$ is log terminal,
and this implies that $W$ is log terminal.
\end{proof}

We close this section with a discussion on surface singularities.
As explained in \cite[Notation~4.1]{KM}, one can define the notions
of {\it numerically log terminal} and {\it numerically log canonical} singularities
for arbitrary normal surfaces, using the perfect pairing
on the relative N\'eron--Severi space of a resolution.
Here we show that a normal surface is log terminal (resp., log canonical)
if and only if it is numerically log terminal (resp.,
numerically log canonical).

\begin{prop}\label{cor:num-lt}
A normal surface $X$ is log terminal
if and only if is numerically log terminal.
\end{prop}

\begin{proof}
By \cite[Proposition~4.11]{KM}, $X$ is
numerically log terminal if and only if it is $\mathbb Q$-factorial and
log terminal. On the other hand, if $X$ is log terminal, then
by Proposition~\ref{prop:lt-lc-reduction-to-log-pairs}
there is a boundary $\D$ such that $(X,\D)$ is log terminal, and hence
numerically log terminal. Again by \cite[Proposition~4.11]{KM},
this implies that $X$ is $\Q$-factorial.
\end{proof}

\begin{prop}\label{prop:num-lc}
A normal surface $X$ is log canonical
if and only if is numerically log canonical.
\end{prop}

\begin{proof}
If $X$ is numerically log canonical, then it is $\Q$-Gorenstein
(cf. \cite[Notation~4.1]{KM}), and hence log canonical.
Conversely, suppose that $X$ is log canonical in the generality introduced
in this section.
We fix a canonical divisor $K_X$ on $X$
and a sufficiently divisible $m \ge 1$. Let $f \colon Y \to X$ be a log
resolution of $(X,\O_X(-mK_X))$,
and write $\O_X(-mK_X) \cdot \O_Y = \O_Y(-A)$ where $A$ is a divisor on $Y$.
Let $K_Y$ be the canonical divisor on $Y$ such that $f_*K_Y = K_X$.
Note that $-A$ is $f$-nef. Let $E = \sum E_i$ be the
reduced exceptional divisor of $f$.
Since $X$ is log canonical, it follows
that if $m$ is sufficiently divisible then the $\Q$-divisor
$$
F := K_{m,Y/X} + E = K_Y + E - \tfrac 1m A
$$
is an effective exceptional $\Q$-divisor.
Let $N = \sum a_i E_i$ be characterized by $K_Y \equiv_f N$. We have
$$
N + E - F \equiv_f \tfrac 1m A.
$$
In particular $-(N+E-F)$ is $f$-nef, and since
it is exceptional, we conclude that $N + E - F \ge 0$ by
the Negativity Lemma (\cite[Lemma 3.39]{KM}). This implies that $a_i \ge -1$
for all $i$, an hence that $X$ is numerically log canonical.
\end{proof}

\begin{cor}
Let $X$ be a normal surface with log terminal (resp., log canonical) singularities.
Then $X$ is $\Q$-factorial (resp., $\mathbb Q$-Gorenstein).
\end{cor}

\section{Terminal and canonical singularities}

In this section we deal with the generalization of canonical and terminal
singularities, and discuss the corresponding extensions of
invariance properties of singularities, plurigenera
and numerical Kodaira dimensions established
in the $\Q$-Gorenstein case in works of Siu, Kawamata and Nakayama.
Throughout the section, the ground field is assumed
to be an algebraically closed field of characteristic zero.

Consider a pair $(X,Z)$, where $X$ is a normal variety and $Z$ is
an effective formal linear combination of proper closed subschemes of $X$.

\begin{defi}
Let $X' \to X$ be a proper birational morphism with $X'$ normal,
and let $F$ be a prime divisor on $X'$.
The {\it log-discrepancy} of a prime divisor $F$ over $X$
with respect to $(X,Z)$ is
$$
a_F(X,Z) := \ord_F(K_{X'/X}) + 1 - \val_F(Z).
$$
The pair $(X,Z)$ is said to be {\it canonical}
(resp., {\it terminal}) if $a_F(X,Z) \ge 1$ (resp., $> 1$)
for every exceptional prime divisor $F$ over $X$.
\end{defi}

Of course these notions coincide with the familiar ones
in the $\Q$-Gorenstein case.
Canonical singularities admit the following characterization
(which is well known in the $\Q$-Gorenstein case).

\begin{prop}\label{prop:char-*-canonical}
Let $X$ be a normal variety, and
suppose that $Z = \sum a_k\.Z_k$ is an effective
formal $\Q$-linear combination
of effective Cartier divisors $Z_k$ on $X$. Then the pair $(X,Z)$ is canonical
if and only if for all sufficiently divisible $m \ge 1$
(in particular, we ask that $m a_k \in \Z$ for every $k$),
and for every log resolution $f \colon Y \to X$ of $\big(X,Z + \O_X(mK_X)\big)$,
there is an inclusion
$$
\O_X\big(m(K_X + Z)\big)\.\O_Y \subseteq \O _Y\big(m(K_Y+Z_Y)\big)
$$
as sub-$\O_Y$-modules of $\K_Y$, where $Z_Y$ is the proper transform of $Z$
(as usual, the canonical divisors $K_X$ and $K_Y$ are
chosen so that $f_*K_Y = K_X$).
\end{prop}

\begin{proof}
Note that $f^{-1}(Z) = f^*Z = - f^*(-Z)$, once we think of $Z$
and $f^{-1}(Z)$ as $\Q$-Cartier $\Q$-divisors.
If $(X,Z)$ is canonical, and $m$ and $f$ are chosen
as in the statement, then we see that
\begin{multline*}
m(K_Y+Z_Y) + f^\natural(-m(K_X + Z)) \ge
m (K_Y+Z_Y) + f^*(-m(K_X + Z))\\
= m (K_Y+Z_Y + f^*(-K_X - Z)) = m(K_{Y/X} +Z_Y - f^{-1}(Z)) \ge 0
\end{multline*}
by Remark~\ref{rmk:comparison-natural-natural} and
Proposition~\ref{prop:val-pullback:additivity:special-case},
and hence we get an inclusion as asserted.
Conversely, suppose that $(X,Z)$ is not canonical, and fix
any log resolution $f\colon Y \to X$ of $(X,Z)$.
Then the $\R$-divisor $K_Y +Z_Y + f^*(-K_X -Z)$ is not effective.
Since $f^*(-K_X-Z)$ is the componentwise limit of the $\Q$-divisors
$\frac 1m\.f^\natural(-m(K_X+Z))$, we can find a sufficiently large
(and divisible) $m$ such that
$K_Y+Z_Y + \frac 1m \. f^\natural(-m(K_X + Z))$ is not effective.
By further blowing up, we may assume that
$f$ is a log resolution of $(X,Z + \O_X(mK_X))$.
Then the assertion follows.
\end{proof}

As an application, we show that deformation
invariance of canonical singularities, plurigenera, and numerical
Kodaira dimension also holds in this more general context.

We start with the extension of Kawamata's theorem on the deformation
invariance of canonical singularities \cite{Kaw3}.

\begin{thm}\label{thm:invariance-can-sings}
Let $f\colon X \to C$ be a flat morphism from a
variety to a smooth curve such that, for some point $0 \in C$, the fiber
$X_0=f^{-1}(0)$ is a normal variety with only canonical singularities. Then
$(X,X_0)$ is canonical in a neighborhood of $X_0$,
and so are all fibers of $f$ over a neighborhood of $0$.
\end{thm}

\begin{proof}
The proof follows the arguments of \cite{Kaw2}.
By shrinking $X$ near $X_0$, we can assume that
$X$ is normal (cf. \cite[Corollary~5.12.7]{Gro}).
We may also assume that $X_0$ is affine.
Let $m>0$ be a sufficiently divisible integer and let
$\mu \colon Y\to X$ be a log resolution of
$(X ,X_0+\O_{X}(mK_{X}))$ which
restricts to a log resolution of
$(X_0,\O_{X_0}\big(mK_{X_0})\big)$.
Let $\{ s_1,\ldots , s_k\}$ be a generating set of sections of
$\O _{X_0}(mK_{X_0})$. Let $Y_0$ be the strict transform of $X_0$.
By Proposition~\ref{prop:char-*-canonical}, there is an inclusion
$$
\O_{X_0}(mK_{X_0})\.\O_{Y_0} \subseteq \O _{Y_0}(mK_{Y_0}).
$$
One sees that there are corresponding
sections $\tilde{s}_i$ of $\O_{Y_0}(mK_{Y_0})$ which push forward to the sections
$s_i$ of $\O _{X_0}(mK_{X_0})$. By Theorem~A of \cite{Kaw2},
after possibly restricting over a neighborhood of $0\in C$, these
sections extend to sections  $\tilde{S}_i$ of $\O _Y(m(K_Y+Y_0))$.
Pushing forward, we obtain sections $S_i$ of
$\O _{X}\big(m(K_{X}+X_0)\big)$
that restrict to $s_i$. It follows by Nakayama's Lemma that the $S_i$
are generators of
$\O _{X}\big(m(K_{X}+X_0)\big)$ at each point of $X_0$.
Thus the inclusion $\mu_*\O _Y\big(m(K_Y+Y_0)\big)\subseteq
\O _{X}\big(m(K_{X}+X_0)\big)$ is an equality in a neighborhood of $X_0$.
Therefore, after restricting to such neighborhood, there is an inclusion
$$
\O _{X}\big(m(K_{X}+X_0)\big)\cdot
\O _Y \subseteq \O _Y \big(m(K_Y+Y_0)\big),
$$
and hence $(X ,X_0)$ is canonical.
\end{proof}

Similarly, we have the following extension of
the invariance of plurigenera (in the general type case)
and of numerical Kodaira dimension
for varieties with canonical singularities
\cite{Siu,Kaw3,Nak}.

\begin{thm}\label{thm:inv-plurigenera-and-num-dim}
Let $f\colon X \to S$ be a projective flat morphism of varieties
whose fibers $X_t = f^{-1}(t)$ are normal varieties
with canonical singularities for every $t \in S$.
Then the following properties hold:
\begin{enumerate}
\item
The numerical Kodaira dimension $\n(X_t)$
is constant on $t \in S$. In particular,
if one fiber $X_0$ is of general type, then
so are the other fibers.
\item
Suppose additionally that the generic fiber $X_\eta$ is
a variety of general type. Then the plurigenera
$P_m(X_t)$ is constant on $t \in S$ for any positive integer $m$.
\end{enumerate}
\end{thm}

\begin{proof}
The proof is similar to those of Theorems~1.3 and~1.2'
of \cite{Kaw2}, after we remark that Kodaira's lemma
(cf. \cite[Proposition~2.2.6]{Laz1}) holds for
(not necessarily Cartier) divisors on a normal projective variety.
\end{proof}

\begin{rmk}\label{rmk:can-v-lt}
Canonical singularities on a $\Q$-Gorenstein normal variety
are obviously purely log terminal. However, it remains unclear
whether an analogous implication still holds
if the singularities are not $\Q$-Gorenstein (cf. Remark~\ref{rmk:K-}).
In fact, in this generality we do not even know if
canonical singularities are
rational (in particular, Cohen--Macaulay) or log canonical.
\end{rmk}

\providecommand{\bysame}{\leavevmode \hbox \o3em
{\hrulefill}\thinspace}

\end{document}